\documentclass[12pt]{amsart}

\usepackage{amssymb}

\newtheorem{theorem}{Theorem}[section]
\newtheorem{proposition}[theorem]{Proposition}
\newtheorem{lemma}[theorem]{Lemma}
\newtheorem{corollary}[theorem]{Corollary}
\newtheorem{definition}[theorem]{Definition}
\newtheorem{remark}[theorem]{Remark}

\numberwithin{equation}{section}

\newcommand{\mb}{\mathbb}

\begin{document}

\baselineskip=16pt

\title{Stability of projective Poincar\'e and Picard bundles}

\author{I. Biswas}

\address{School of Mathematics, Tata Institute of
Fundamental Research, Homi Bhabha Road, Bombay 400005, India}

\email{indranil@math.tifr.res.in}

\author{L. Brambila-Paz}

\address{CIMAT, Apdo. Postal 402, C.P. 36240. Guanajuato, Gto,
M\'exico}

\email{lebp@cimat.mx}

\author{P. E. Newstead}

\address{Department of Mathematical Sciences, The University of
Liverpool, Peach Street, Liverpool, L69 7ZL, England}

\email{newstead@liverpool.ac.uk}

\subjclass[2000]{14H60, 14J60}

\date{\today}

\thanks{All authors are members of the international research group VBAC. The second author acknowledges the support of CONACYT grant 48263-F}
\begin{abstract}
Let $X$ be an irreducible smooth projective curve of genus $g\ge3$ defined over
the complex numbers and let ${\mathcal M}_\xi$
denote the moduli space of stable vector bundles on $X$ of rank $n$
and determinant $\xi$, where $\xi$ is a fixed line bundle
of degree $d$. If $n$ and
$d$ have a common divisor, there is no universal
vector bundle on $X\times {\mathcal M}_\xi$. We prove that there is a projective bundle on 
$X\times {\mathcal M}_\xi$ with the property that its restriction to
$X\times\{E\}$ is isomorphic to $P(E)$ for all $E\in\mathcal{M}_\xi$ and that
this bundle (called the projective Poincar\'e bundle) is stable with respect
to any polarization; moreover its restriction to $\{x\}\times\mathcal{M}_\xi$
is also stable for any $x\in X$. We prove also stability results for bundles
induced from the projective Poincar\'e bundle by homomorphisms
$\text{PGL}(n)\to H$ for any reductive $H$. We show further that there is a
projective Picard bundle on a certain open subset $\mathcal{M}'$ of
$\mathcal{M}_\xi$ for any $d>n(g-1)$ and that this bundle is also stable. We
obtain new results on the stability of the Picard bundle even when $n$ and $d$
are coprime.\end{abstract}

\maketitle

\section{Introduction}

Let $X$ be an irreducible smooth projective curve over $\mathbb C$
of genus $g\geq 3$. For any integer $n\ge2$ and any
algebraic line bundle $\xi$ of degree $d$ on
$X$, let ${\mathcal M}_\xi$ denote the moduli space of stable
vector bundles $E$ of rank $n$ and degree $d$ on $X$ with $\det
E:=\bigwedge^n E\, \cong\, \xi$.

If $d$ is coprime to $n$, there is a universal
vector bundle $\mathcal{U}$
on $X\times {\mathcal M}_\xi$. The direct image of
$\mathcal{U}$ on ${\mathcal M}_\xi$,
which we will denote by $\mathcal{W}$, is called a Picard bundle;
we need to assume $d\geq 2n(g-1)$ in order to
ensure that the Picard sheaf is locally free.
The stability of $\mathcal{U}$ and  $\mathcal{W}$
was proved in \cite{BBN} and \cite{BBGN} respectively.
Moreover, for any point $x\in X$, semistability of the
restriction of $\mathcal{U}$ to $\{x\}\times
{\mathcal M}_\xi$ was proved in \cite{BBN}, while its stability
was established in \cite{LN}.

In this paper we will consider the situation where
$n$ and $d$ have a common divisor. In this case
there is no universal vector bundle on $X\times {\mathcal
M}_\xi$ (see \cite{Ramanan}; also \cite{Ne}). However, there is a
\textit{projective Poincar\'e bundle}
$$
 \mathcal{PU}\, \longrightarrow\, X\times {\mathcal
M}_\xi
$$
such that, for any point $E \in {\mathcal M}_\xi$, the
restriction of $\mathcal{PU}$ to $X\times\{E\}\subset X\times
{\mathcal M}_\xi$ is isomorphic to the projective bundle $P(E)$
over $X$ that parametrizes all lines in the fibres of the stable
vector bundle $E$ on $X$ (see section \ref{sec2} for more details). For any $x\in X$, we denote by $\mathcal{PU}_x$ 
the restriction of $\mathcal{PU}$ to $\{x\}\times\mathcal{M}_\xi$. Although $\mathcal{U}$ does not 
exist, one can also define an {\em adjoint Poincar\'e bundle} $\text{ad}\,\mathcal{U}$. In fact, 
since there is a bijective correspondence between projective bundles of fibre dimension $n-1$ and 
principal $\text{PGL}(n)$-bundles, one can associate 
with any homomorphism of algebraic groups $\rho: \text{PGL}(n)\to H$ with $H$ reductive an {\em induced 
Poincar\'e principal $H$-bundle} $\mathcal{U}^\rho$. When $H=\text{GL}(m)$, one can regard $\mathcal{U}^\rho$ 
as a vector bundle on $X\times\mathcal{M}_\xi$; the adjoint Poincar\'e bundle is a special case of this construction.

We construct also the projective Picard bundle
$\mathcal{PW}$ on the Zariski-open subset $\mathcal{M}'$ of ${\mathcal M}_\xi$ 
consisting of those $E\in\mathcal{M}_\xi$ for which $H^1(X,E)=0$. The fibre of
$\mathcal{PW}$ over any point $E \in {\mathcal M}'$ is
identified with the projective space $P(H^0(X,E))$. 
The construction requires no restriction on $d$, but $\mathcal{PW}=\emptyset$ for $d\le n(g-1)$.

The moduli space ${\mathcal M}_\xi$ is an irreducible smooth
quasiprojective variety defined over $\mathbb C$ of dimension
$(n^2-1)(g-1)$. It is the smooth locus of the moduli space
$\overline{\mathcal M}_\xi$ of (S-equivalence classes of) semistable vector bundles on $X$
of rank $n$ and determinant $\xi$ (recall that $g\geq3$).
The variety $\overline{\mathcal M}_\xi$ is locally factorial (see \cite{DN}) and
$$
\text{Pic}({\mathcal M}_\xi)\cong\text{Pic}(\overline{\mathcal M}_\xi)\cong{\mathbb Z}.
$$
In particular, $\overline{\mathcal M}_\xi$ has a unique polarization represented by a divisor $\Theta$. 
It follows that there is a unique notion of (slope-)stable vector bundle on ${\mathcal
M}_\xi$. The notion of semistable and
stable vector bundles extends to principal bundles (see
\cite{Ra}, \cite{RR}, \cite{RS}, \cite{AB} for the definitions) and in particular to projective bundles; we shall give a direct definition 
for projective bundles in section  \ref{sec2}.

In section \ref{sec3} we prove the following results on the stability of the projective Poincar\'e bundle.

\noindent{\bf Theorem \ref{thm1}.}
{\em Let $X$ be a smooth projective algebraic curve of genus $g\ge3$, $n\ge2$ an integer and $\xi$ 
a line bundle on $X$ of degree $d$. Let ${\mathcal M}_\xi$ denote the moduli space of stable vector bundles on $X$ of rank $n$ and determinant $\xi$ and let $\mathcal{PU}$ be the projective Poincar\'e bundle on 
$X\times {\mathcal M}_\xi$. Then $\mathcal{PU}_x$ is stable for all $x\in X$.}

\noindent{\bf Theorem \ref{thm2}.}
{\em Under the hypotheses of Theorem \ref{thm1}, $\mathcal{PU}$ is stable with respect to any polarization on $X\times {\mathcal M}_\xi$.}

The proofs involve Hecke transformations (see \cite{NR0}, \cite{NR}) and use the same constructions as in \cite{BBGN}. Since  the concept of stability for projective bundles 
agrees with that for principal $\text{PGL}(n)$-bundles (see Remark \ref{remnew4}), these theorems can be restated in terms of principal bundles.
Using a theorem of \cite{RR} concerning principal bundles and recalling that a homomorphism $G\to H$ with 
$H$ reductive is {\em irreducible} if its image is not contained in any proper parabolic subgroup, we obtain

\noindent{\bf Theorem \ref{thm4}.}
{\em Under the hypotheses of Theorem \ref{thm1}, let ${\rm PGL}(n)\to H$ be a homomorphism with $H$ reductive and let  $\mathcal{U}^\rho$ be the induced Poincar\'e $H$-bundle. Then
\begin{itemize}
\item[(i)] $\mathcal{U}^\rho_x$ is semistable for all $x\in X$;
\item[(ii)] $\mathcal{U}^\rho$ is semistable with respect to any polarization on $X\times\mathcal{M}_\xi$;
\item[(iii)] if $\rho$ is irreducible, then $\mathcal{U}^\rho$ is stable with respect to any polarization on $X\times\mathcal{M}_\xi$.
\end{itemize}}

For the last part of this theorem, we need to show that $P(E)^\rho$ is stable for general 
$E\in\mathcal{M}_\xi$. We offer two proofs of this. The first (see Lemma
\ref{lemnew6}) is an algebraic argument based on the concept of monodromy as
introduced in \cite{BPS}. The second (see Remark \ref{remnew14}) involves an 
argument of Subramanian \cite{S} using unitary representations. 
When $n=2$ and $\rho$ is the adjoint representation, we give a third  proof of 
the theorem (Theorem \ref{thm5}) using the methods of the present paper. 

In section \ref{sec-stable}, we define the Picard bundle $\mathcal{PW}$ and prove

\noindent{\bf Theorem \ref{thm3}.}
{\em Under the hypotheses of Theorem \ref{thm1}, suppose further that
  $d>n(g-1)$. Then the projective Picard bundle $\mathcal{PW}$ on ${\mathcal
    M}'$ is stable.}

When $\gcd(n,d)=1$, we can define a Picard sheaf $\mathcal{W}_\xi$ on
$\mathcal{M}_\xi$ whose restriction to $\mathcal{M}'$ is a vector bundle
$\mathcal{W}'$ such that $P(\mathcal{W}')\cong\mathcal{PW}$. As a corollary of
Theorem \ref{thm3} (Corollary \ref{newcor3}), we obtain the stability of
$\mathcal{W}_\xi$ and $\mathcal{W}'$, thus extending the result of
\cite{BBGN}.

\noindent{\bf Notation.} We shall consistently write $\mathcal E_x$ (respectively $\mathcal{P}_x$) for the restriction of a vector bundle ${\mathcal E}$ (respectively projective or principal bundle $\mathcal P$) on $X\times Z$ to $\{x\}\times Z$. We suppose throughout that $X$ is a smooth irreducible projective algebraic curve of genus $g\ge3$ defined over $\mathbb{C}$.

\section{The projective Poincar\'e bundle}\label{sec2}

We begin with a definition of stability for a projective bundle.

Let $Y$ be a polarized irreducible locally factorial projective variety and let $Z$ be a Zariski-open subset of $Y$ whose complement has codimension $\ge2$ in $Y$. Fix a divisor $D$ on $Y$ defining the polarization. For any vector bundle $E$ on $Z$, we can define $c_1(E)$ as a divisor class on $Y$ and write
$$\deg E=[c_1(E).D^{m-1}](Y),$$
where $m=\dim Y$. Now let $P$ be a projective bundle on $Z$ and let $P'$ be a projective subbundle of the restriction of $P$ to a Zariski-open subset $Z'$ of $Z$ whose complement has codimension $\ge2$. Write  $q$, $q'$ for the projections of $P$, $P'$ to $Z$, $Z'$ respectively. We have an exact sequence of vector bundles 
\begin{equation}\label{enew1}
0\longrightarrow T^{\text{rel}}_{q'} P'
\longrightarrow T^{\text{rel}}_qP|_{Z'}
\longrightarrow N_1 \longrightarrow 0
\end{equation}
on $P'$, where $T^{\text{rel}}_{q'}P'$
and $T^{\text{rel}}_qP_Z$ are the relative tangent bundles and $N_1$ is the normal bundle. 
The direct image $N:=q'_*(N_1)$ is a vector bundle on $Z'$. (The higher direct images are all $0$ by (\ref{enew1}) and the fact that $H^i({\mathbb P},T{\mathbb P})=0$ for $i\ge1$ and any projective space ${\mathbb P}$.)
\begin{definition}\label{def1}\begin{em}
The projective bundle $P$ is {\em stable} ({\em semistable}) if, for every such $P'$, the condition
$$
\deg N> 0\ \ (\deg N\ge0)
$$
holds.\end{em}\end{definition}

\begin{remark}\label{remnew4}\begin{em}
We have adopted this form of the definition because it is the most convenient for our purposes. 
Moreover, if $P$ is the projectivization of a vector bundle
$V$ on $Z$,
then a projective subbundle $P'$ defines a subbundle $V'$
of $V$ over $Z'$. In that case, the bundle $N$ is identified
with $\textit{Hom}(V'\, , V/V'))\, =\,
(V')^*\bigotimes (V/V')$. Using this it follows
immediately that $P$ is stable (semistable) if and only if the
vector bundle $V$ is stable (semistable). Definition  \ref{def1} is also equivalent to the 
standard definition of stability for a principal $\text{PGL}(n)$-bundle (see \cite[Definition 4.7]{RR} or \cite{RS}). 
To see this, note that, if we denote also by $P$ the $\text{PGL}(n)$-bundle corresponding to $P$, the subbundle $P'$ of our definition corresponds to a reduction of structure group 
$\sigma: Z'\to P/Q$, where $Q$ is a maximal parabolic subgroup of $\text{PGL}(n)$. If $T_{P/Q}$ denotes the 
tangent bundle along the fibres of $P/Q$, then $\sigma^*(T_{P/Q})$ is 
isomorphic to the bundle $N$ of Definition \ref{def1}; now compare 
\cite[Definition 1.1 and Lemma  2.1]{Ra}, which are stated for curves 
but generalize immediately to higher dimension by requiring $\sigma$ to be 
defined on the complement of a subvariety of codimension $\ge2$.
\end{em}\end{remark}

Recall that the standard definition of ${\mathcal M}_\xi$ is as a quotient $\pi:R\to{\mathcal M}_\xi$ 
of a Zariski-open subset $R$ of a Quot-scheme $Q$ by a free action of 
$\text{PGL}(M)$ for some $M$ (see \cite{Ses} or \cite[Chapter 5]{New2}). In fact 
$Q$ is a closed subset of $\text{Quot}({\mathcal O}^M;P)$, the Grothendieck scheme of quotients of 
${\mathcal O}^M$ with Hilbert polynomial $P$, for some positive integer $M$ and 
polynomial $P$. There is a natural action of $\text{GL}(M)$ on $Q$ which 
descends to an action of 
$\text{PGL}(M)$ and $R$ is a Zariski-open 
$\text{PGL}(M)$-invariant subset of $Q$; the restriction of the action 
of $\text{PGL}(M)$ to $R$ is free and defines the 
quotient $\pi$. There also exists a universal quotient on $X\times Q$ 
to which the action of 
$\text{GL}(M)$ lifts naturally, but $\lambda I$ acts by multiplication by 
$\lambda$, so the action does not descend to $\text{PGL}(M)$. The universal 
quotient restricts to a vector bundle on $X\times R$, which, after tensoring by 
the pullback of some bundle ${\mathcal O}(-m)$ on $X$, becomes a vector bundle 
${\mathcal E}_R$ such that ${\mathcal E}_R|_{X\times\{r\}}$ is the stable 
bundle $\pi(r)$ for all $r\in R$. As indicated above, $\text{GL}(M)$ acts on this 
bundle with $\lambda I$ acting by multiplication by $\lambda$, so 
$\text{PGL}(M)$ acts on the associated projective bundle 
$P({\mathcal E}_R)$. The quotient $\mathcal{PU}:=P({\mathcal E}_R)/
\text{PGL}(M)$ is then a projective bundle  whose restriction to 
$X\times\{E\}$ is isomorphic to $P(E)$ for all $E\in{\mathcal M}_\xi$. 
The uniqueness of $\mathcal{PU}$ as constructed in this way is a corollary 
of the following result which we shall need later.

\begin{proposition}\label{propnew1}
Let $\mathcal E$ be a vector bundle on $X\times Z$ such that the restriction of $\mathcal E$ to $X\times\{z\}$ is stable of rank $n$ and determinant $\xi$ for all $z\in Z$ and let $\phi_{\mathcal E}:Z\to\mathcal{M}_\xi$ be the corresponding morphism. Then the projective bundles $P({\mathcal E})$ and $(\text{id}_X\times \phi_{\mathcal E})^*(\mathcal{PU})$ are isomorphic.
\end{proposition} 

\begin{proof} We have a pullback diagram
\begin{equation}\label{enew3}
\begin{matrix}
Y\ \ &\stackrel{\phi_Y}{\longrightarrow}& R\\
\Big\downarrow\pi' && \ \ \Big\downarrow\pi\\
Z\ \ &\stackrel{\phi_{\mathcal E}}{\longrightarrow}&\ \ {\mathcal M}_\xi.\end{matrix}
\end{equation}
The vector bundles $(\text{id}_X\times\phi_Y)^*{\mathcal E}_R$ and 
$(\text{id}_X\times\pi')^*{\mathcal E}$ 
on $X\times Y$ have the property that their restrictions to $X\times\{y\}$ are 
stable and isomorphic for all $y\in Y$. If we denote by $p_Y:X\times Y\to Y$ 
the natural projection, it follows that 
$$p_{Y*}(Hom((\text{id}_X\times\phi_Y)^*{\mathcal E}_R),
(\text{id}_X\times\pi')^*{\mathcal E})$$
is a line bundle ${\mathcal L}$ on $Y$, and there is a natural isomorphism
\begin{equation}\label{e600}
(\text{id}_X\times\phi_Y)^*{\mathcal E}_R\otimes p_Y^*{\mathcal L}
\longrightarrow (\text{id}_X\times\pi')^*{\mathcal E}.
\end{equation}
Moreover $\text{GL}(M)$ acts naturally on both $(\text{id}_X\times\phi_Y)^*{\mathcal E}_R$ and 
$(\text{id}_X\times\pi')^*{\mathcal E}$; in the first case, $\lambda I$ acts 
by multiplication by $\lambda$, in the second by the identity. There is also 
a natural action of $\text{GL}(M)$ on ${\mathcal L}$ and an induced action on 
the left-hand side of (\ref{e600}) which descends to $\text{PGL}(M)$. In particular 
(\ref{e600}) is $\text{PGL}(M)$-equivariant and the same holds for the 
corresponding isomorphism of projective bundles. Now take quotients.
\end{proof}

\begin{corollary}\label{cornew1}
Suppose that $\pi':R'\to\mathcal{M}_\xi$ defines $\mathcal{M}_\xi$ as a quotient of $R'$ by a free 
action of ${\rm PGL}(M')$ and
\begin{itemize}
\item $\mathcal{E}_{R'}$ is a vector bundle on $X\times R'$ such that 
$\mathcal{E}_{R'}|_{X\times\{r'\}}$ is the stable bundle $\pi'(r')$ for all $r'\in R'$;
\item the action of ${\rm PGL}(M')$ lifts to 
$P(\mathcal{E}_{R'})$.
\end{itemize}
Then $P(\mathcal{E}_{R'})/{\rm PGL}(M')\cong\mathcal{PU}$.
\end{corollary}
\begin{proof} Apply Proposition \ref{propnew1} to ${\mathcal E}_{R'}$ to get an isomorphism 
$$P(\mathcal{E}_{R'})\cong\left(\text{id}_X\times\phi_{\mathcal{E}_{R'}}\right)
^*(\mathcal{PU}).$$ 
It follows from the proof of the proposition that the isomorphism can be chosen to be 
$\text{PGL}(M')$-equivariant. Now take quotients.
\end{proof}

In view of this, we shall call $\mathcal{PU}$ the {\em projective Poincar\'e bundle} on $X\times{\mathcal M}_\xi$. 
It should be noted that it is not the same as the universal projective bundle constructed in \cite{BBNN}, which exists on a certain open set in the moduli space of projective bundles with the appropriate topological invariants. This open set 
is a quotient by a finite group of the Zariski-open set in $\mathcal{M}_\xi$ constructed in the following lemma.

\begin{lemma}\label{prop0}
There is a non-empty Zariski-open subset $\mathcal Z$ of ${\mathcal
M}_\xi$ such that, for each stable vector bundle $E \in
\mathcal{Z}$, the corresponding projective bundle $P(E)$ does not
admit any nontrivial automorphism.
\end{lemma}

\begin{proof}
Let $E \in {\mathcal M}_\xi$ be such that the associated
projective bundle $P(E)$ admits a nontrivial
automorphism
$
\tau' : P(E) \longrightarrow P(E)
$.
The automorphism $\tau'$ gives an isomorphism of vector bundles
\begin{equation}\label{e3}
\tau : E \longrightarrow E\otimes L ,
\end{equation}
where $L$ is some line bundle of degree $0$. From the given
condition that $\tau'$ is nontrivial it follows that $L \not\cong
{\mathcal O}_X.$ Taking the top exterior product of both sides of
\eqref{e3} we conclude that $L^{\otimes n}\cong {\mathcal O}_X$.

Suppose now that $\tau$ exists with $L$ of order $r\ge2$ as an element of the Jacobian $J(X)$. 
Choose an isomorphism $L^{\otimes r}\stackrel{\cong}{\longrightarrow}\mathcal{O}_X$ and let $s_r$ be the section of $L^{\otimes r}$ corresponding to the constant section $1$ of $\mathcal{O}_X$. Via this isomorphism $\tau^r$ 
defines an automorphism of $E$, which has the form $\lambda\,\text{id}_E$ since $E$ is stable. 
Let $\sigma:Y\to X$ be the cyclic covering defined as the subvariety of the total space of $L$ given by the 
equation $t^r-\lambda s_r=0$. Then $E$ is the direct image of a vector bundle $V$ on $Y$ of rank $\frac{n}{r}$ and degree $d$; moreover $V$ is necessarily stable (see \cite{NR1} for details of the construction; also \cite[Example 3.4 and 
Proposition 3.6]{BNR} for the case $r=n$). Note that $\sigma$ is 
determined by $L$ and that there are only finitely many choices for $L$ (up to isomorphism). Since $V$ depends on 
$\left(\frac{n}{r}\right)^2(g(Y)-1)+1$ parameters, it follows that the stable vector bundles $E$ of determinant $\xi$ arising in this way depend on at most
$\nu$ parameters, where
\begin{eqnarray*}
\nu&=&\frac{n^2}{r^2}(g(Y)-1)+1-g\\
&=&\frac{n^2}{r^2}r(g-1)+1-g\\
&=&\left(\frac{n^2}{r}-1\right)(g-1)<(n^2-1)(g-1)=\dim\mathcal{M}_{\xi}.
\end{eqnarray*}
This completes the proof.
\end{proof}

As an immediate consequence we have

\begin{corollary}\label{cor1}
The projective Poincar\'e bundle $\mathcal{PU}$ on $X\times
{\mathcal M}_\xi$ does not admit any nontrivial
automorphism.\qed\end{corollary}

\begin{remark}\label{remnew1}\begin{em}
A precise description of the variety of stable bundles $E$ for which 
$E\cong E\otimes L$ for a fixed $L$ is given in \cite[Proposition 3.3]{NR1}. 
For an analytic proof of Lemma \ref{prop0}, see \cite[Theorem 1.2 and Proposition 1.6]{Li}.
An algebraic proof in the coprime case is given in \cite[Proposition 3.8]{BBN}. In the 
topologically trivial case (i.e. $d$ is divisible by $n$), the lemma is also a special case 
of \cite[Proposition 2.6]{BBNN}.
\end{em}\end{remark}

\begin{remark}\label{remnew2}\begin{em}
If $E\cong E\otimes L$, then we have a non-zero homomorphism $L\to \text{End}\,E$. 
Now suppose that $E$ is stable. Since we are in characteristic $0$, $\text{End}\,E$ is semistable of degree $0$ (see \cite{RR}), 
so this homomorphism embeds $L$ as a subbundle of $\text{End}\,E$. 
Now $\text{End}\,E\cong\mathcal{O}\oplus\text{ad}\,E$. Hence, if $L\not\cong\mathcal{O}$, $L$ embeds as a subbundle of $\text{ad}\,E$. We deduce from Lemma \ref{prop0} and its proof that, if $E\in\mathcal{M}_\xi$, then the following conditions are equivalent:
\begin{itemize}
\item $E\in\mathcal{Z}$, where $\mathcal{Z}$ is the Zariski-open subset of $\mathcal{M}_\xi$ consisting of those $E$ for which $E\cong E\otimes L\Rightarrow L\cong\mathcal{O}$; 
\item $P(E)$ admits no non-trivial automorphism;
\item $\text{ad}\,E$ possesses no line subbundle of degree $0$
\end{itemize} 
(see also \cite[Proposition 1.6]{Li}, \cite[Proposition 3.10]{BBN}).
\end{em}\end{remark}

\begin{remark}\label{remnew11}\begin{em} It follows from Remark \ref{remnew2} (see \cite[Corollary 3.11]{BBN}) that, if 
$E\in\mathcal{Z}$ and $n=2$, then $\text{ad}\,E$ is stable. In fact, more generally, we have
\end{em}\end{remark}

\begin{lemma}\label{lemnew6} For any $n$ and any irreducible homomorphism
  $\rho:{\rm PGL}(n)\to H$ with $H$ reductive, the principal $H$-bundle $P(E)^\rho$ is stable for general $E\in\mathcal{M}_\xi$.
\end{lemma}
\begin{proof}

We recall the concept of monodromy introduced in
\cite{BPS}. For a stable $G$-bundle $E_G$ on $X$, its monodromy is a reductive subgroup
of $G$ (see \cite[Lemma 4.13]{BPS}; any irreducible subgroup is
automatically reductive). Hence all stable $G$-bundles whose monodromy is a proper
subgroup of $G$ admit reduction of structure group
to some proper reductive subgroup of $G$. There are countably many proper
reductive subgroups $G'$ of $G$ up to
conjugation, for each of which 
\begin{eqnarray*}
\dim M_X(G)&=&\dim G(g-1)+ \dim(\text{centre of }G)\\ &>& \dim G'(g-1)+
\dim(\text{centre of }G')=\dim M_X(G'),\end{eqnarray*}
where $M_X(G)$ (resp. $M_X(G')$) denotes the moduli space of stable $G$-bundles
(resp. $G'$-bundles).
Therefore, all stable $G$-bundles with monodromy a proper
subgroup of $G$ are contained in a countable union of subvarieties.

In our case, we take $G=\text{PGL}(n)$ and deduce that there exists a bundle $E\in\mathcal{M}_{\xi}$ such
that the monodromy of $P(E)$ is $\text{PGL}(n)$. It follows from \cite[Lemma
4.13]{BPS} that $P(E)^\rho$ is stable for this $E$ and hence for general $E$
since stability is an open property.
\end{proof}

\begin{remark}\label{remnew14}\begin{em}
For an analytic proof of this lemma, recall that the principal $H$-bundle given by any irreducible unitary representation of the fundamental group is stable \cite[Theorem 7.1]{Ra}  (here {\em unitary} means that the image of the representation is contained in a maximal compact subgroup). The result follows by an argument of Subramanian \cite[\S3]{S} (see also \cite[proof of Theorem 2.7]{BBN}). \end{em}\end{remark} 

\begin{remark}\label{remnew}\begin{em}
For any homomorphism $\rho:\text{PGL}(n)\to H$, the induced Poincar\'e $H$-bundle $\mathcal{U}^\rho$ on $X\times\mathcal{M}_\xi$ has the property that $\mathcal{U}^\rho|_{X\times\{E\}}$ is isomorphic to $P(E)^\rho$ for all $E\in\mathcal{M}_\xi$. In particular $\text{ad}\,\mathcal{U}|_{X\times\{E\}}\cong\text{ad}\,E$. The bundle $\text{ad}\,\mathcal{U}$ can also be constructed by noting that the action of $\text{PGL}(M)$ on $R$ lifts to an action on $\text{ad}\,\mathcal{E}_R$, which therefore descends to a bundle on $X\times\mathcal{M}_\xi$. This bundle coincides with $\text{ad}\,\mathcal{U}$.
\end{em}\end{remark}

\section{Stability of Poincar\'e bundles}\label{sec3} 

In this section, we shall prove our results on the stability of Poincar\'e bundles. We begin by recalling two constructions from 
\cite{BBGN}.

Let $x \in X$ and let $F$ be a vector bundle over $X$ of rank $n$ such that
$\det F\cong\xi(x)$. Let
${\mathbb P}:=P(F_x^*)$
be the projective space parametrizing the hyperplanes in the
fibre $F_x$. Let
$
p:X\times {\mathbb P} \, \longrightarrow\,
X
$
be the projection and
$
\iota\, :\, {\mathbb P}\, \hookrightarrow\, X\times {\mathbb P}
$
the inclusion map defined by $z\, \longmapsto\,(x\, ,z)$.

We have the following diagram of homomorphisms of sheaves on
$X\times {\mathbb P}$
\begin{equation}\label{e7}
\begin{matrix}
&& 0 && 0\\
&& \Big\downarrow && \Big\downarrow\\
&& p^*(F(-x)) &= &p^*(F(-x))\\
&& \Big\downarrow && \Big\downarrow\\
0&\longrightarrow & {\mathcal E} &\longrightarrow & p^*F
&\longrightarrow & \iota_*{\mathcal O}_{\mathbb P}(1)
&\longrightarrow & 0\\
&& \Big\downarrow && \Big\downarrow && \Vert\\
0&\longrightarrow & \iota_*(\Omega^1_{\mathbb P}(1)) &\longrightarrow & F_x \otimes_{\mathbb C}
\iota_*{\mathcal O}_{\mathbb P} &\longrightarrow & \iota_*{\mathcal
O}_{\mathbb P}(1)
&\longrightarrow & 0\\
&& \Big\downarrow && \Big\downarrow\\
&& 0 && 0\\
\end{matrix}
\end{equation}
(see \cite[p. 565, (5)]{BBGN}). 

\begin{lemma}\label{lemnew1}
There exists an exact sequence of vector bundles
\begin{equation}\label{e8}
0\, \longrightarrow\, {\mathcal O}_{\mathbb P}(1)\, \longrightarrow\,{\mathcal E}_x\, \longrightarrow\,
\Omega^1_{\mathbb P}(1)\, \longrightarrow\, 0
\end{equation}
on $\mathbb P$.
\end{lemma}
\begin{proof}
Pulling back the left hand column of (\ref{e7}) by $\iota$, we get
$$0\longrightarrow\Omega^1_{\mathbb P}(1)\longrightarrow F_x\otimes_{\mathbb C}\mathcal{O}_{\mathbb P}
\longrightarrow\mathcal{E}_x\longrightarrow\Omega^1_{\mathbb P}(1)\longrightarrow0.$$
This splits into two short exact sequences
\begin{equation}\label{eqnew11}0\longrightarrow\Omega^1_{\mathbb P}(1)\longrightarrow F_x\otimes_{\mathbb C}\mathcal{O}_{\mathbb P}
\longrightarrow\mathcal{K}\longrightarrow0
\end{equation}
and
\begin{equation}\label{eqnew12}0\longrightarrow\mathcal{K}
\longrightarrow\mathcal{E}_x\longrightarrow\Omega^1_{\mathbb P}(1)\longrightarrow0,
\end{equation}
where $\mathcal{K}$ is a line bundle on $\mathbb{P}$. Since $\deg(\Omega^1_{\mathbb P}(1))=-1$, it follows from \eqref{eqnew11} that $\deg\mathcal{K}=1$. Now \eqref{eqnew12} gives the result.
\end{proof}

\begin{lemma}\label{lem1}
Let $W\, \subset\, {\mathcal E}_x$ be a nonzero coherent subsheaf
of the vector bundle ${\mathcal E}_x$ in \eqref{e8} such that
\begin{itemize}
\item the quotient ${\mathcal E}_x/W$ is torsion-free, and
\item $\frac{{\rm deg}(W)}{{\rm rk}(W)}\, \geq\,
\frac{{\rm deg}({\mathcal E}_x)}{{\rm rk}({\mathcal E}_x)}$.
\end{itemize}
Then $W$ contains the line subbundle ${\mathcal O}_{\mathbb P}(1)$ 
of ${\mathcal E}_x$ in
\eqref{e8}.
\end{lemma}

\begin{proof}
Note that, by \eqref{e8}, $\deg(\mathcal{E}_x)=0$. Let
$
W_1\, \subseteq\, W
$
be the first term of the Harder--Narasimhan filtration of
the subsheaf $W$. So $W_1$ is semistable, and
$$
\frac{\text{deg}(W_1)}{\text{rk}(W_1)}\, \geq\,
\frac{\text{deg}(W)}{\text{rk}(W)}\, .
$$
Since ${\rm deg}(W)/{\rm rk}(W)\, \geq\, {\rm deg}({\mathcal E}_x)/{\rm
rk}({\mathcal E}_x)$, we have
\begin{equation}\label{e9}
\frac{\text{deg}(W_1)}{\text{rk}(W_1)} \, \geq \, \frac{{\rm
deg}({\mathcal E}_x)}{{\rm rk}({\mathcal E}_x)}\, =\, 0\, >\,
\frac{1}{1-n}\, =\, \frac{{\rm deg}(\Omega^1_{\mathbb P}(1))}{{\rm rk}(\Omega^1_{\mathbb P}(1))}\, .
\end{equation}

The vector bundle $\Omega^1_{\mathbb P}(1)$ is stable (see \cite[Chapter II 
Theorem 1.3.2]{OSS}). Therefore, from \eqref{e9} it follows that there is no
nonzero homomorphism from $W_1$ to $\Omega^1_{\mathbb P}(1)$. Consequently, $W_1$ is contained in
the line subbundle ${\mathcal O}_{\mathbb P}(1)$ of ${\mathcal E}_x$ in \eqref{e8}. 
Since $\mathcal{E}_x/W_1$ is torsion-free, this
completes the proof of the lemma.
\end{proof}

We shall apply this construction in the case where $F$ is 
$(0,1)$-stable (see \cite[Definition 8.1]{NR0} or \cite[Definition 5.1]{NR}). 
For convenience, we recall this definition: a vector bundle $F$ on 
$X$ is $(0,1)$-{\it stable} if, for any proper subbundle $F'$ of $F$,
$$\frac{\deg F'}{\text{rk}F'}<\frac{\deg F-1}{\text{rk}F}.$$ 
We have (\cite[Lemma 1]{BBGN} 
or \cite[Lemma 5.5]{NR})

\begin{lemma}\label{lemnew3}
Let $F$ be a $(0,1)$-stable vector bundle of rank $n$ and 
determinant $\xi(x)$. Then the vector bundle $\mathcal{E}$ in 
\eqref{e7} is a family of stable vector bundles of rank $n$ and determinant $\xi$ on X.\qed
\end{lemma}

Let $\mathcal{M}_{\xi(x)}$ denote the moduli space of stable vector bundles on $X$ of rank $n$ and determinant $\xi(x)$. If $F\in\mathcal{M}_{\xi(x)}$ is $(0,1)$-stable, Lemma \ref{lemnew3} gives us a morphism
$$\phi_F:P(F_x^*)\to\mathcal{M}_\xi.$$
We have 
\begin{lemma}\label{lemnew4}
For any $(0,1)$-stable $F\in\mathcal{M}_{\xi(x)}$, the morphism 
$\phi_F$ is an embedding. Moreover the locus of $(0,1)$-stable vector bundles $F\in\mathcal{M}_{\xi(x)}$ is a non-empty Zariski-open subset.
\end{lemma}
\begin{proof}
For the first part, see \cite[Lemma 3]{BBGN} or \cite[Lemma 5.9]{NR}. For the second part, see \cite[Lemma 2]{BBGN} and note that, in the last line of the proof, we do not require $n$ and $d$ to be coprime since $g\ge3$.
\end{proof}

The second construction is the reverse of the one just considered.

For any vector bundle $E$ on $X$ of rank $n$ and determinant
$\xi$, take a line $\ell\, \subset\, E_x$ in the fibre of $E$ over
$x$. Let $F$ be the vector bundle over $X$ that fits in the
following short exact sequence of sheaves
\begin{equation}\label{e11}
0\, \longrightarrow\, F(-x)
\, \longrightarrow\, E\, \longrightarrow\,
E_x/\ell\, \longrightarrow\, 0\, .
\end{equation}

Consider the subset $H_x$ of $\mathcal{PU}_x$ consisting of pairs $(E,\ell)$ for which the bundle $F$ defined by \eqref{e11} is $(0,1)$-stable. We have a diagram
\begin{equation}\begin{matrix}\label{enew4}
H_x&\stackrel{p}{\longrightarrow}&\mathcal{M}_\xi\\q\Big\downarrow\\
\mathcal{M}_{\xi(x)}
\end{matrix}\end{equation}
given by $p(E,\ell)=E$ and $q(E,\ell)=F$. The map $p$ is the restriction of the projection 
$\mathcal{PU}_x\to\mathcal{M}_\xi$ and its image is the set of bundles $E\in\mathcal{M}_\xi$ 
for which there exists a line $\ell\in E_x$ such that 
the vector bundle $F$ in \eqref{e11} is $(0,1)$-stable. 

\begin{lemma}\label{lemnew5}\ 
\begin{itemize}
\item[(i)] $H_x$ is non-empty and Zariski-open in $\mathcal{PU}_x$ and $q$ is a morphism;
\item[(ii)] $p(H_x)$ is non-empty and Zariski-open in $\mathcal{M}_\xi$.\end{itemize}\end{lemma}
\begin{proof}
(This is proved in the coprime case in \cite[p. 566]{BBGN} but the proof uses a Poincar\'e 
vector bundle, so we give full details here.)

(i) We return to the construction of $\mathcal{M}_\xi$ as a quotient\linebreak $\pi:R\to\mathcal{M}_\xi$. The bundle $P(\mathcal{E}_R)_x$ on $R$ parametrizes a family of pairs $(E,l)$, where $E\in\mathcal{M}_\xi$ and $\ell$ is a line in $E_x$, and hence parametrizes sequences \eqref{e11}; the subset $H_x'$ of $P(\mathcal{E}_R)_x$ for which $F$ is $(0,1)$-stable is therefore Zariski-open and is non-empty by Lemmas \ref{lemnew3} and \ref{lemnew4}.  The group $\text{PGL}(M)$ acts on $P(\mathcal{E}_R)_x$ with $H_x'$ as an invariant subset. We can now identify $\mathcal{PU}_x$ and $H_x$ with 
$P(\mathcal{E}_R)_x/\text{PGL}(M)$ and $H_x'/\text{PGL}(M)$. 
Moreover the map $H_x'\to\mathcal{M}_{\xi(x)}$ given by sending a 
point of $H_x'$ to the corresponding $F$ is a morphism by the universal 
property of $\mathcal{M}_{\xi(x)}$; hence $q$ is a morphism.

(ii) By (i), $p(H_x)$ is non-empty. It is also Zariski-open since 
$p$ is the restriction of the projection morphism 
$\mathcal{PU}_x\to{\mathcal M}_\xi$.\end{proof}

We are now ready to state and prove the first of our main theorems.

\begin{theorem}\label{thm1}
Let $X$ be a smooth projective algebraic curve of genus $g\ge3$, $n\ge2$ an integer and $\xi$ 
a line bundle on $X$ of degree $d$. Let ${\mathcal M}_\xi$ denote the moduli space of stable vector bundles on $X$ of rank $n$ and determinant $\xi$ and let $\mathcal{PU}$ be the projective Poincar\'e bundle on 
$X\times {\mathcal M}_\xi$. Then $\mathcal{PU}_x$ is stable for all $x\in X$.
\end{theorem}
\begin{proof} Let $P'$ be a projective subbundle of the restriction of $\mathcal{PU}_x$ to 
a Zariski-open subset $Z'$ of $\mathcal{M}_\xi$ with complement of codimension $\ge2$. 
By \eqref{enew4} and Lemma \ref{lemnew5}, $p^{-1}(Z')$ is a Zariski-open subset 
of $H_x$ whose complement $S$ has codimension $\ge2$, in other words
$$\dim S\le\dim\mathcal{M}_\xi+n-3.$$
Since the image of $q$ has dimension $\dim\mathcal{M}_\xi$, the intersection 
of $S$ with the general fibre of $q$ is a closed subset of dimension $\le n-3$. It follows from this, taking account of Lemma \ref{lemnew4}, that there exists a $(0,1)$-stable bundle $F\in\mathcal{M}_{\xi(x)}$ such that $\phi_F^{-1}(Z')$ has complement of codimension $\ge2$. Moreover $F$ is defined by a pair $(E,\ell)$ and we can suppose (by generality of $F$) that $\ell\in P(E_x)$ 
is not in the fibre of $P'$ over $E$.

By Proposition \ref{propnew1}, $(\text{id}_X\times\phi_F)^*(\mathcal{PU})
\cong P(\mathcal{E})$, where $\mathcal{E}$ is defined in \eqref{e7}. 
Moreover $\phi_F^*(P')$ lifts to a vector subbundle $V'$ of 
$\mathcal{E}_x':=\mathcal{E}_x|_{P(F_x^*)\cap\phi_F^{-1}(Z')}$. It follows that, if $N$ is defined as in 
Definition \ref{def1}, then $\phi_F^*N\cong V'^*\otimes(\mathcal{E}_x'/V')$. Now take 
$W=V'$ in Lemma \ref{lem1} (strictly speaking, we take $W$ to be an extension of $V'$ to a coherent subsheaf of $\mathcal{E}_x$). The condition on $\ell$ means that $V'$ does not contain the 
line subbundle $\mathcal{O}_{\mathbb{P}}(1)|_{P(F_x^*)\cap\phi_F^{-1}(Z')}$ of $\mathcal{E}_x'$. Hence
$$\frac{\deg V'}{\text{rk}V'}<\frac{\deg\mathcal{E}_x}{\text{rk}\mathcal{E}_x}$$
or equivalently
$$\deg\phi_F^*N>0.$$
It follows from Lemma \ref{lemnew4} that $\phi_F^*(\Theta)\cong\mathcal{O}(\delta)$ for some $\delta>0$; hence $\deg N>0$, which completes the proof.
\end{proof}

In order to prove the stability of $\mathcal{PU}$, we need a further lemma.

\begin{lemma}\label{lemmaproj} Let $P$ be a projective bundle on $X\times{\mathcal
M}_\xi$ such that the following two conditions hold:
\begin{itemize}
\item for general $x\, \in\, X$, the projective bundle $P_x$ on $\mathcal{M}_\xi$ is semistable, and
\item for general $E\, \in\, {\mathcal M}_\xi$, the
projective bundle $P(E)$ on $X$ is stable.
\end{itemize}
Then $P$ is stable with respect to any
polarization on $X\times{\mathcal M}_\xi$.
\end{lemma}
\begin{proof}
We follow the proof of \cite[Lemma 2.2]{BBN}.

Since
${\mathcal M}_\xi$ is unirational, there is no
nonconstant map from ${\mathcal M}_\xi$ to
$\text{Pic}^\delta(X)$. It follows from the see--saw theorem (see
\cite[p. 54, Corollary 6]{Mu}) that any line bundle over
$X\times{\mathcal M}_\xi$ is of the form $L_1\boxtimes L_2$. 
Hence any polarization of $X\times\mathcal{M}_\xi$ can be represented by a divisor of the 
form $a\alpha+b\Theta$, where $\alpha$ is a fixed ample divisor on $X$ and $a$, $b$ are positive integers.

Now suppose that $P'$ is a projective subbundle of the restriction of 
$P$ to a Zariski-open subset $Z'$ of $X\times\mathcal{M}_\xi$ with 
complement of codimension $\ge2$. Then for general 
$E\in\mathcal{M}_\xi$, we have $Z'\supset X\times\{E\}$, and for 
general $x\in X$, the complement 
of $Z'\cap\{x\}\times\mathcal{M}_\xi$ in $\mathcal{M}_\xi$ has codimension 
$\ge2$. With the notation of Definition \ref{def1}, we have, from the 
hypotheses of the lemma,
\begin{equation}\label{enew2}
\deg N_x\ge0,\ \ \deg N|_{X\times\{E\}}>0.
\end{equation}
Now
$$\deg N=[c_1(N).(a\alpha+b\Theta)^m](X\times\mathcal{M}_\xi),$$
where $m:=\dim\mathcal{M}_\xi$. So
\begin{eqnarray*}
\deg N&=&[c_1(N).(mab^{m-1}\alpha\Theta^{m-1}+b^m\Theta^m)](X\times\mathcal{M}_\xi)\\
&=&mab^{m-1}\alpha(X).\deg N_x+b^m\deg(N|_{X\times\{E\}})\Theta^m(\mathcal{M}_\xi)\\
&>&0
\end{eqnarray*}
by \eqref{enew2}.
\end{proof}

\begin{remark}\label{remnew13}\begin{em}
(i) Lemma \ref{lemmaproj} is true for vector bundles by \cite[Lemma 2.2]{BBN}.

(ii) The lemma can be generalized to $H$-bundles; for the proof, we simply replace $P'$ by $\sigma:Z'\to P/Q$, where $P$ is an $H$-bundle and $Q$ is a maximal parabolic subgroup, and $N$ by 
$\sigma^*(T_{P/Q})$.
\end{em}\end{remark}

Combining Lemma \ref{lemmaproj} and Theorem \ref{thm1} we now have our second theorem.

\begin{theorem}\label{thm2}
Under the hypotheses of Theorem \ref{thm1}, $\mathcal{PU}$ is stable with respect to any polarization on $X\times {\mathcal M}_\xi$. \qed
\end{theorem}

\begin{theorem}\label{thm4}
Under the hypotheses of Theorem \ref{thm1}, let $\rho:{\rm PGL}(n)\to H$ be a homomorphism with $H$ reductive and let  $\mathcal{U}^\rho$ be the induced Poincar\'e $H$-bundle. Then
\begin{itemize}
\item[(i)] $\mathcal{U}^\rho_x$ is semistable for all $x\in X$;
\item[(ii)] $\mathcal{U}^\rho$ is semistable with respect to any polarization on $X\times\mathcal{M}_\xi$;
\item[(iii)] if $\rho$ is irreducible, then $\mathcal{U}^\rho$ is stable with respect to any polarization on $X\times\mathcal{M}_\xi$.
\end{itemize}
\end{theorem}

\begin{proof}
(i) and (ii) follow from Theorems \ref{thm1} and \ref{thm2} and \cite[Theorem 3.18]{RR}. (As stated, 
\cite[Theorem 3.18]{RR} applies only to bundles defined on a smooth curve, but it is extended in \cite[\S4]{RR} to the case of a Zariski-open subset with complement of codimension $\ge2$ in a smooth projective variety $Y$. In fact, the proof still works if we assume only that $Y$ is locally factorial.)

In view of Lemma \ref{lemnew6} and Remark \ref{remnew13}(ii), (iii) follows from (i).
\end{proof}

In the important special case of the adjoint Poincar\'e bundle, we can give an algebraic proof of Theorem \ref{thm4} when $n=2$ using the techniques of the present paper. We state this as

\begin{theorem}\label{thm5}
Under the hypotheses of Theorem \ref{thm1}, suppose further that $n=2$. Let $\text{ad}\,\mathcal{U}$ be the adjoint Poincar\'e bundle on $X\times\mathcal{M}_\xi$. Then
\begin{itemize}
\item[(i)] $\text{ad}\,\mathcal{U}_x$ is semistable for all $x\in X$;
\item[(ii)] $\text{ad}\,\mathcal{U}$ is stable with respect to any polarization on $X\times\mathcal{M}_\xi$.
\end{itemize}
\end{theorem}

\begin{proof} For (i), we use the same technique as in the proof of Theorem \ref{thm1}. Since $n=2$, we have $\mb{P}=\mb{P}^1$ and \eqref{e8} becomes
$$0\longrightarrow\mathcal{O}_{\mb{P}}(1)\longrightarrow\mathcal{E}_x\longrightarrow\mathcal{O}_{\mb{P}}(-1)\longrightarrow0,$$
which splits to give
$$\mathcal{E}_x\cong\mathcal{O}_{\mb{P}}(1)\oplus\mathcal{O}_{\mb{P}}(-1).$$
Hence
$$\text{End}\,\mathcal{E}_x\cong\mathcal{O}_{\mb{P}}(2)\oplus
\mathcal{O}_{\mb{P}}^{\oplus2}\oplus\mathcal{O}_{\mb{P}}(-2).$$
Note moreover that $\mathcal{O}_{\mb{P}}(2)$ is contained in the subbundle 
$\text{ad}\,\mathcal{E}_x$ of endomorphisms of trace $0$, so 
\begin{equation}\label{enew5}
\text{ad}\,\mathcal{E}_x\cong\mathcal{O}_{\mb{P}}(2)\oplus
\mathcal{O}_{\mb{P}}\oplus\mathcal{O}_{\mb{P}}(-2).
\end{equation}
The bundle $F$ in \eqref{e7} is determined by a pair $(E,\ell)$ with $\ell$ a line in the fibre $E_x$. The subspace $\mathcal{O}_{\mb{P}}(2)_E$ of $\text{ad}\,E_x$ consists of the 
endomorphisms of $E_x$ of the form
$$E_x\longrightarrow E_x/\ell\longrightarrow \ell\hookrightarrow E_x.$$
As $\ell$ varies, these endomorphisms span the $3$-dimensional space $\text{ad}\,E_x$.

Now suppose that $V$ is a proper subbundle of $\text{ad}\,\mathcal{U}_x$ 
defined over some open set in $\mathcal{M}_\xi$ whose complement has 
codimension at least 2. As in the first part of the proof of 
Theorem \ref{thm1}, for a general choice of $(E,\ell)$, the bundle 
$V':=\phi_F^*(V)$ is defined on the whole of $P(F_x^*)$ (note that 
$\dim P(F_x^*)=1$). In view of the argument of the previous paragraph, 
we can suppose further that $V'$ does not contain the subbundle 
$\mathcal{O}_{\mb{P}}(2)$ of $\text{ad}\,\mathcal{E}_x$. It follows from \eqref{enew5} that $V'$ is isomorphic to a subsheaf of $\mathcal{O}_{\mb{P}}\oplus\mathcal{O}_{\mb{P}}(-2)$ and so
$$\deg V'\le0.$$
Since also $\deg(\text{ad}\,\mathcal{E}_x)=0$, we have
$$\deg (V'^*\otimes(\text{ad}\,\mathcal{E}_x/V'))\ge0$$
and the proof of (i) is completed in the same way as for Theorem \ref{thm1}.

For (ii), note that, if $E\in\mathcal{M}_\xi$ is general, then $\text{ad}\,E$ is stable by Lemma \ref{lemnew6}.  In view of Remark \ref{remnew13}(i), this completes the proof. 
\end{proof}

\section{Stability of the projective Picard bundle}\label{sec-stable}

In order to define the projective Picard bundle, we return to the construction of $\mathcal{M}_\xi$ as a quotient $\pi:R\to\mathcal{M}_\xi$. Let $p_R:X\times R\to R$ denote the natural projection. Then $p_{R*}\mathcal{E}_R$ is a torsion-free sheaf on $R$ and is non-zero if and only if $d>n(g-1)$. The subset $R'$ of $R$ defined by
$$R':=\{r\in R|H^1(X,\mathcal{E}_R|_{X\times\{r\}})=0\}$$
is Zariski-open and invariant under the action of $\text{PGL}(M)$. We write 
$$\mathcal{M}':=\{E\in\mathcal{M}_\xi| H^1(X,E)=0\}=\pi(R').$$
\begin{lemma}\label{lemnew7}\ 

\begin{itemize}
\item[(i)] If $d>n(g-1)$, then $\mathcal{M}'$ is non-empty and has complement of codimension $\ge2$ in $\mathcal{M}_\xi$.
\item[(ii)] If $d\ge2n(g-1)$, then $\mathcal{M}'=\mathcal{M}_\xi$.
\end{itemize}
\end{lemma}

\begin{proof} (i) Using Serre duality, the complement of $\mathcal{M}'$ in 
$\mathcal{M}_\xi$ can be identified with the intersection of the Brill-Noether 
locus $B(n,d_1,1)$ in ${\mathcal M}(n,d_1)$ with
$\mathcal{M}_{K^{\otimes n}\otimes\xi^{-1}}$, where $d_1=n(2g-2)-d$ and 
$$B(n,d_1,1):=\{E\in{\mathcal M}(n,d_1)| h^0(E)\ge1\}.$$ We claim that this 
Brill-Noether locus has the 
expected codimension in $\mathcal{M}_{K^{\otimes n}\otimes\xi^{-1}}$, namely
$$1-(2n(g-1)-d)+n(g-1)=1+d-n(g-1)\ge2.$$
The fact that $B(n,d_1,1)$ has the expected codimension in 
${\mathcal M}(n,d_1)$ is standard, but we have not been able to locate a proof 
for the intersection with $\mathcal{M}_{K^{\otimes n}\otimes\xi^{-1}}$ in the literature, 
so we give a proof here. Any stable bundle $E$ of degree $d_1$ with $h^0(E)\ge1$ 
can be expressed in the form 
\begin{equation}\label{e601}
0\longrightarrow L'\longrightarrow E\longrightarrow F\longrightarrow 0,
\end{equation}
where $L'$ is an effective line bundle of degree $d'\ge0$ and $F$ is a 
vector bundle of rank $n-1$ and degree $d_1-d'>0$. Since $E$ is stable, so is 
$E^*\otimes L'$, and by (\ref{e601}) we have $\deg(E^*\otimes L')<0$; it follows 
(again by stability) that $h^0(E^*\otimes L')=0$ and hence $h^0(F^*\otimes L')=0$.
Now, for any fixed $L'$, the bundles $F$ in (\ref{e601}) form a bounded family 
dependent on at most $\left((n-1)^2-1\right)(g-1)$ parameters (the corresponding 
fact for bundles of non-fixed determinant is proved in \cite[Lemma 4.1]{BGN} 
and the same proof works for fixed determinant). Hence the bundles $E$ in 
(\ref{e601}) depend on at most
\begin{equation*}\begin{array}{lll}
d'&+&\left((n-1)^2-1\right)(g-1)+\dim h^1(F^*\otimes L')-1\\
&=&d'+\left((n-1)^2-1\right)(g-1)+d_1-d'-(n-1)d'+(n-1)(g-1)-1\\
&=&(n^2-1)(g-1)-(1-d_1+n(g-1))-(n-1)d'
\end{array}\end{equation*}
parameters. So $B(n,d_1,1)\cap\mathcal{M}_{K^{\otimes n}\otimes\xi^{-1}}$ has codimension at least
$$1-d_1+n(g-1)=1+d-n(g-1)$$
as required. From standard Brill-Noether theory, this is also the maximum 
possible codimension, proving our claim.

(ii) is clear since any stable bundle $E$ of rank $n\ge2$ and degree $d\ge2n(g-1)$ has $H^1(X,E)=0$.
\end{proof}

Suppose now that $d>n(g-1)$. Since $\dim H^0(X,\mathcal{E}_R|_{X\times\{r\}})$ 
is constant on $R'$, the direct image $p_{R'*}(\mathcal{E}_R|_{X\times R'})$ 
is a vector bundle. The action of $\text{PGL}(M)$ on $R'$ lifts to an action of 
$\text{GL}(M)$ on ${\mathcal E}_R|_{X\times R'}$, hence also
on $p_{R'*}(\mathcal{E}_R|_{X\times R'})$. Since $\lambda I$ acts by 
multiplication by $\lambda$, this descends to an action of $\text{PGL}(M)$ on 
$P(p_{R'*}(\mathcal{E}_R|_{X\times R'}))$; we write
$$\mathcal{PW}:=P(p_{R'*}(\mathcal{E}_R|_{X\times R'}))/\text{PGL}(M).$$

We have the following analogue of Proposition \ref{propnew1}:

\begin{proposition}\label{propnew2}
Let $\mathcal E$ be a vector bundle on $X\times Z$ such that the restriction
of $\mathcal E$ to $X\times\{z\}$ is stable of rank $n$ and determinant $\xi$
for all $z\in Z$ and let $\phi_{\mathcal E}:Z\to\mathcal{M}_\xi$ be the
corresponding morphism. Suppose in addition that
$H^1(X,\mathcal{E}|_{X\times\{z\}})=0$ for all $z\in Z$, so that
$\phi_{\mathcal{E}}$ factors through a morphism $\phi_{\mathcal{E}}':Z\to\mathcal{M}'$. Then the projective bundles $P(p_{Z*}{\mathcal E})$ and $(\phi_{\mathcal E}')^*(\mathcal{PW})$ are isomorphic.
\end{proposition} 

\begin{proof}
Consider the pullback diagram
\begin{equation}\label{enew8}
\begin{matrix}
Y'\ \ &\stackrel{\phi_{Y'}}{\longrightarrow}& R'\\
\Big\downarrow\pi' && \ \ \Big\downarrow\pi\\
Z\ \ &\stackrel{\phi_{\mathcal E}'}{\longrightarrow}&\ \ {\mathcal M}'.\end{matrix}
\end{equation}
obtained by restricting \eqref{enew3} to $\mathcal{M}'$. The restriction of 
(\ref{e600}) to $X\times Y'$ yields a $\text{PGL}(M)$-equivariant 
isomorphism
$$(\text{id}_X\times\phi_{Y'})^*(\mathcal{E}_R|_{X\times R'})\otimes p_{Y'}^*{\mathcal L}'
\cong(\text{id}_X\times\pi')^*\mathcal{E}$$
of  vector bundles on $X\times Y'$, where ${\mathcal L}'$ is a line bundle on $Y'$ 
admitting an action of $\text{GL}(M)$. Taking direct images, 
we obtain a $\text{PGL}(M)$-equivariant isomorphism 
$$p_{Y'*}((\text{id}_X\times\phi_{Y'})^*(\mathcal{E}_R|_{X\times R'}))\otimes {\mathcal L}'\cong 
p_{Y'*}((\text{id}_X\times\pi')^*\mathcal{E})$$
of vector bundles on $Y'$. Taking quotients of the corresponding projective bundles by the action of $\text{PGL}(M)$ gives the result.
\end{proof}

The proof of Corollary \ref{cornew1} now applies to show that $\mathcal{PW}$ is independent of the choice of $R$ and $\pi$.
We therefore call $\mathcal{PW}$ the \textit{projective Picard bundle}. It is
a projective bundle with fibre dimension $d-n(g-1)-1$. 

\begin{remark}\label{remnew5}\begin{em}
If $\gcd(n,d)=1$ and $d\ge2n(g-1)$, $\mathcal{PW}$ is  the
projectivization of the Picard bundle $\mathcal{W}_\xi$ on $\mathcal{M}_\xi$
associated to the universal vector bundle on $X\times\mathcal{M}_\xi$. When
$d>2n(g-1)$, the stability of $\mathcal{W}_\xi$ was proved in
\cite{BBGN}. Without the restriction on $d$, we can still define a {\em Picard
  sheaf} $\mathcal{W}_\xi$ on $\mathcal{M}_\xi$ by
$$\mathcal{W}_\xi:=p_*\mathcal{U}_\xi,$$
where $\mathcal{U}_\xi$ is a Poincar\'e bundle on $X\times\mathcal{M}_\xi$ and
$p:X\times\mathcal{M}_\xi\to\mathcal{M}_\xi$ is the natural projection. The
sheaf $\mathcal{W}_\xi$ is determined up to tensoring by a line bundle on
$\mathcal{M}_\xi$ and has rank $\max\{0,d-n(g-1)\}$. Moreover
$\mathcal{W}_\xi$ is torsion-free and its restriction to $\mathcal{M}'$ is a
vector bundle $\mathcal{W}'$ with the property that $P(\mathcal{W}')\cong\mathcal{PW}$.
\end{em}\end{remark}

\begin{theorem}\label{thm3}
Under the hypotheses of Theorem \ref{thm1}, suppose further that $d>n(g-1)$. Then the projective Picard bundle $\mathcal{PW}$ on ${\mathcal M}'$ is stable.
\end{theorem}

\begin{proof} 
Suppose that $P'$ is a projective subbundle of the restriction of 
$\mathcal{PW}$ to some Zariski-open subset $Z'$ of $\mathcal{M}'$ with 
complement of codimension $\ge2$. As in the proof of Theorem \ref{thm1}, for 
general $(E,\ell)$, the complement of $Z'':=\phi_F^{-1}(Z')$ has codimension 
$\ge2$ in $\mathbb{P}:=P(F_x^*)$. We can also suppose that the fibre of $P'$ 
at $E$ contains a point representing a section $s\in H^0(E)$ such that 
$s(x)\not\in\ell$.  Now consider the direct image of \eqref{e7} by the natural 
projection $p_2:X\times Z''\to Z''$; since 
$R^1p_{2*}(\mathcal{E}|_{X\times Z''})=0$ by definition of $\mathcal{M}'$, 
we obtain the following diagram of
exact sequences on $Z''$:

\begin{equation}\begin{array}{ccccccccc}\label{d8}
&&0&&0\\
&&\Big\downarrow&&\Big\downarrow\\
&&H^0(X,F(-x))\otimes\mathcal{O}_{Z''}&=&H^0(X,F(-x))\otimes\mathcal{O}_{Z''}\\
&&\Big\downarrow&&\Big\downarrow\\
0&\longrightarrow&p_{2*}(\mathcal{E}|_{X\times Z''}) &\longrightarrow&
H^0(X,F)\otimes\mathcal{O}_{Z''}& \longrightarrow&
\mathcal{O}_{Z''}(1)&\longrightarrow&0\\
&&\Big\downarrow&& \Big\downarrow&&\Arrowvert\\
0&\longrightarrow
&\Omega^1_{Z''}(1)&\longrightarrow&F_x\otimes\mathcal{O}_{Z''}&\longrightarrow&
\mathcal{O}_{Z''}(1)&\longrightarrow&0
\end{array}\end{equation}

By Proposition \ref{propnew2}, $P(p_2^*(\mathcal{E}|_{X\times Z''}))$ is isomorphic to $\phi_F^*(\mathcal{PW})|_{Z''}$ and we can therefore lift $\phi_F^*(P')|_{Z''}$ to a vector subbundle $V'$ of $p_{2*}(\mathcal{E}|_{X\times Z''})$. By construction, the induced homomorphism
\begin{equation}
V'\to\Omega^1_{Z''}(1)\label{110}\end{equation}
has non-zero image. Since  $\Omega^1_{Z''}(1)$ is stable of degree
$-1$, the image of (\ref{110}) has degree $\le-1$. Moreover the
kernel of (\ref{110}) is a subsheaf of a trivial sheaf and therefore
has degree $\le0$. So
$$\deg V'\le-1.$$
On the other hand, by $(\ref{d8})$, the bundle
 $p_{2*}(\mathcal{E}|_{X\times Z''})$ has degree $-1$; so
$$
\deg(V'^*\otimes (p_{2*}(\mathcal{E}|_{X\times Z''})/V'))>0.$$
This completes the proof in the same way as that of Theorem \ref{thm1}.\end{proof}

\begin{corollary}\label{newcor3}Suppose that $\gcd(n,d)=1$ and
  $d>n(g-1)$. Then
\begin{itemize}
\item[(i)] the Picard sheaf $\mathcal{W}_\xi$ on $\mathcal{M}_\xi$ is stable;
\item[(ii)] the Picard bundle $\mathcal{W}'$ on $\mathcal{M}'$ is stable.
\end{itemize}\end{corollary}
\begin{proof}
Since $\mathcal{M}_\xi\setminus\mathcal{M}'$ has codimension $\ge2$ in
$\mathcal{M}_\xi$, the two statements are equivalent. The result now follows
from the fact that $P(\mathcal{W}')=\mathcal{PW}$ and Theorem \ref{thm3}.
\end{proof}


\begin{thebibliography}{AAAAA}

\bibitem{AB} B. Anchouche and I. Biswas, Einstein--Hermitian
connections on polystable principal bundles over a compact
K\"ahler manifold, \textit{Amer. J. Math.} \textbf{123}
(2001), 207--228.

\bibitem{BBNN} V. Balaji, I. Biswas, D. S. Nagaraj and
P. E. Newstead, Universal families on moduli spaces of
principal bundles on curves, \textit{Int. Math. Res. Not.} (2006),
Article Id 80641.

\bibitem{BBN} V. Balaji, L. Brambila-Paz and P. E.
Newstead, Stability of the Poincar\'e bundle,
\textit{Math. Nachr.} \textbf{188} (1997), 5--15.

\bibitem{BNR} A. Beauville, M. S. Narasimhan and S.
Ramanan, Spectral curves and the generalised theta divisor,
\textit{J. Reine Angew. Math.} \textbf{398} (1989), 169--179.

\bibitem{BBGN} I. Biswas, L. Brambila-Paz,
T. L. G\'omez and P. E. Newstead, Stability of the Picard
bundle, \textit{Bull. London Math. Soc.} \textbf{34} (2002),
561--568.

\bibitem{BPS} I. Biswas, A. J. Parameswaran and S. Subramanian,
Monodromy group for a strongly semistable principal bundle over a curve. 
\textit{Duke Math. J.} \textbf{132} (2006), 1--48

\bibitem{BGN} L. Brambila-Paz, I. Grzegorczyk and P. E. Newstead, Geography of 
Brill-Noether loci for small slopes, \textit{J. Algebraic Geom.} \textbf{6} 
(1997), 645--669.


\bibitem{DN} J.-M. Drezet and M. S. Narasimhan, Groupe de 
Picard des vari\'et\'es de modules de fibr\'es semi-stables sur 
les courbes alg\'ebriques, \textit{Invent. Math.} \textbf{97} (1989),
53--94.

\bibitem{LN} H. Lange and P. E. Newstead, On Poincar\'e
bundles of vector bundles on curves, \textit{Manuscr. Math.}
\textbf{117} (2005), 173--181.

\bibitem{Li} Y. Li, Spectral curves, theta divisors and Picard bundles,
\textit{Internat. J. Math.} \textbf{2} (1991), 525--550.

\bibitem{Mu} D. Mumford, \textit{Abelian varieties,} Oxford
University Press, 1970.

\bibitem{NR0}M. S. Narasimhan and S. Ramanan,
Deformations of the moduli space of vector bundles over an algebraic curve,
\textit{Ann. of Math.} \textbf{101} (1975), 391--417.

\bibitem{NR1} M. S. Narasimhan and S. Ramanan,
Generalised Prym varieties as fixed points, \textit{J. Indian Math. Soc.},
\textbf{39} (1975), 1--19. 

\bibitem{NR} M. S. Narasimhan and S. Ramanan,
Geometry of Hecke cycles. I, \textit{C. P. Ramanujam--a tribute},
pp. 291--345, Tata Inst. Fund. Res. Studies in Math., 8, Springer,
Berlin-Heidelberg-New York, 1978.

\bibitem{Ne} P. E. Newstead, A non-existence theorem for
families of stable bundles, \textit{J. London Math. Soc.}
\textbf{6} (1973), 259--266.

\bibitem{New2} P. E. Newstead, \textit{Introduction to moduli problems and 
orbit spaces}, Tata Inst. Fund. Res. Lectures on Math., 51, 
Springer-Verlag, Berlin-Heidelberg-New York, 1978.

\bibitem{OSS} Ch. Okonek, M. Schneider and H. Spindler, 
\textit{Vector bundles on complex projective spaces}, Progress in Mathematics, 
3, Birkh\"auser, Boston, 1980.

\bibitem{Ramanan} S. Ramanan, The moduli spaces of vector
bundles over an algebraic curve, \textit{Math. Ann.}
\textbf{200} (1973), 69--84.

\bibitem{RR} S. Ramanan and A. Ramanathan, Some remarks on the instability flag,
\textit{T\^ohoku Math. J.} \textbf{36} (1984), 269--291.

\bibitem{Ra} A. Ramanathan, Stable principal bundles
on a compact Riemann surface, \textit{Math. Ann.}
\textbf{213} (1975), 129--152.

\bibitem{RS} A. Ramanathan and S. Subramanian,
Einstein--Hermitian connections
on principal bundles and stability, \textit{J. Reine
Angew. Math.} \textbf{390} (1988), 21--31.

\bibitem{Ses}
C. S. Seshadri (with J.-M. Drezet), Fibr\'es vectoriels sur les courbes alg\'ebriques, 
\textit{Ast\'erisque} \textbf{96} (1982).

\bibitem{S} S. Subramanian,
Mumford's example and a general construction, 
\textit{Proc. Indian Acad. Sci. Math. Sci.} \textbf{99} (1989), 197--208.

\end{thebibliography}
\end{document}